\documentclass[11pt,reqno]{article}
\usepackage{graphicx}
\usepackage{amsmath}
\usepackage{amsthm}
\usepackage{latexsym,bm}
\usepackage{amssymb}
\usepackage{indentfirst}
\newtheorem{thm}{Theorem}[section]

\newtheorem{lem}[thm]{Lemma}
\newtheorem{prop}[thm]{Proposition}

\newtheorem{Remark}{Remark}
\numberwithin{equation}{section}

\begin{document}

\title{Riemannian submersions from compact four manifolds}

\author{Xiaoyang Chen \footnote{The author is supported in part by NSF DMS-1209387.}}
 \date{}
\maketitle

\begin{abstract}
We show that under certain conditions, a nontrivial Riemannian submersion from positively curved four manifolds
does $not$ exist. This gives a partial answer to a conjecture due to Fred Wilhelm.
We also prove a rigidity theorem for Riemannian submersions with
totally geodesic fibers from compact four-dimensional Einstein manifolds.
\end{abstract}

\section{Introduction}
A smooth map $\pi: (M, g)\rightarrow (N,h)$ is
a Riemannian submersion if $\pi_*$ is surjective and satisfies the following property:
$$g_p(v,w)=h_{\pi(p)}(\pi_*v, \pi_*w)$$
for any $v,w$ that are tangent vectors in $TM_p$ and perpendicular to the kernel of $\pi_*$.

\vskip0.1in

A fundamental problem in Riemannian geometry is to study the interaction between curvature and topology.
A lot of important work has been done in this direction.  In this paper we study a similar problem for Riemannian submersions:

\vskip0.1in

\noindent \textbf{Problem:} {\itshape Explore the structure of $\pi$ under additional curvature assumptions of $(M,g)$.}

\vskip0.1in

When $(M,g)$ has constant sectional curvature, we have the following classification results (\cite{gromoll1988low}, \cite{walschap1992metric}, \cite{wi2001}).

\begin{thm}

Let $\pi: (M^m, g)\rightarrow (N,h)$ be a nontrivial  Riemannian submersion (i.e.  $0< dim N< dim M$) with connected fibers, where $(M^m,g)$ is compact and has constant sectional curvature $c$.

\vskip0.1in

1. If $c<0$, then there is no such Riemannian submersion.

\vskip0.1in

2. If $c=0$, then locally $\pi$ is the projection of a metric product onto one of its factors.

\vskip0.1in

3. If $c>0$ and $M^m$ is simply connected , then $\pi$ is metrically congruent to the Hopf fibration, i.e,
there exist isometries $f_1: M^m \rightarrow \mathbb{S}^{m}$ and
$f_2: N \rightarrow \mathbb{P}(\mathbb{K})$ such that $pf_1=f_2\pi$,
where $p$ is the standard projection from $\mathbb{S}^{m}$ to projective spaces $\mathbb{P}(\mathbb{K})$.
\end{thm}

However, very little is known about the structure of $\pi$  if $(M,g)$ is not of constant curvature. In this paper we consider two different curvature conditions:

\vskip0.1in

1. $(M,g)$ has positive sectional curvature.

\vskip0.1in

2. $(M,g)$ is an Einstein manifold.

\vskip0.1in

 When $(M,g)$ has positive
 sectional curvature, we have the following important conjecture due to Fred Wilhelm (although never published anywhere).

\vskip0.1in

\noindent \textbf{Conjecture 1} {\itshape Let $\pi: (M,g)\rightarrow (N,h)$ be a nontrivial Riemannian submersion, where $(M,g)$ is
a compact Riemannian manifold with positive sectional curvature. Then $dim (F)< dim (N)$, where $F$ is the fiber
of $\pi$.}

\vskip0.1in

By Frankel's theorem \cite{Frankel}, it is not hard to see that Conjecture $1$ is true if at least two fibers of $\pi$ are totally geodesic.
In fact, since any two fibers do not intersect with each other, Frankel's theorem implies that $2$ $ dim (F)< dim (M)$.
 Hence $dim (F)< dim (N)$. If all fibers of $\pi$ are totally geodesic, we have the following much stronger result:

\begin{prop} \label{totally}
Let $\pi: (M,g)\rightarrow (N,h)$ be a nontrivial Riemannian submersion such that all fibers of $\pi$ are totally geodesic, where $(M,g)$ is
a compact Riemannian manifold with positive sectional curvature. Then $dim (F)< \rho(dim (N))+1$, where $F$ is any fiber
of $\pi$ and $\rho(n)$ is the maximal number of linearly independent vector fields on $S^{n-1}$.
\end{prop}

Notice that we always have $\rho(dim (N))+1\leq dim (N)-1+1=dim (N)$ and equality holds if and only $dim (N)=2,4$ or $8$.

\vskip0.1in

Although not explicitly stated, Proposition \ref{totally} appears in \cite{Fatness}. For completeness, we will give a proof in section $3$.

\vskip0.2in

When $dim(M)=4$, Conjecture $1$ is equivalent to the following conjecture.

\vskip0.1in

\noindent \textbf{Conjecture 2} {\itshape There is no nontrivial Riemannian submersion from any compact four manifold $(M^4,g)$ with positive sectional curvature. }

\vskip0.1in

In fact, suppose there exists such a Riemannian submersion $\pi: (M^4,g)\rightarrow (N,h)$. Then Conjecture $1$ would imply
$dim (N)=3$. Hence the Euler number of $M^4$ is zero.
On the other hand, since $(M^4,g)$ has positive sectional curvature, $H^1(M^4, \mathbb{R})=0$ by Bochner's vanishing theorem (\cite{Pet}, page $208$).
By Poincar$\acute{e}$ duality, the Euler number of $M^4$ is positive. Contradiction.

\vskip0.1in

 Let $\pi: (M,g)\rightarrow (N,h)$ be a Riemannian submersion. We say that a function $f$ defined on $M$ is basic if $f$ is constant along each fiber.
 A vector field $X$ on $M$ is basic if it is horizontal and is $\pi$-related to a vector field on $N$.
 In other words, $X$ is the horizontal lift of some vector field on $N$.
 Let $H$ be the mean curvature vector field of the fibers  and $A$ be the O'Neill tensor of $\pi$. 
 We denote by $\|A\|$ the norm of $A$, i.e.,  $\|A\|^2=\sum_{i,j}\|A_{X_i} X_j\|^2$, where $\{{X_i}\}$ is a local orthonormal basis of the horizontal distribution of $\pi$.
The next theorem  gives a partial answer to Conjecture $2$.

\begin{thm} \label{positive}
 There is $no$ nontrivial Riemannian submersion from any compact four manifold with positive sectional curvature such that
 either $\|A\|$ or $H$ is basic. 
\end{thm}

We emphasize that in Conjecture $1$ the assumption that $(M,g)$ has positive sectional curvature can $not$ be replaced by
$(M,g)$ has positive sectional curvature $almost$ everywhere, namely, $(M,g)$ has nonnegative sectional curvature everywhere
and has positive sectional curvature on an open and dense subset of $M$.
Indeed, Let $g$ be the metric on $S^2 \times S^3$
constructed by B. Wilking which has positive sectional curvature $almost$ everywhere \cite{wilking2002manifolds}.
Then by a theorem of K. Tapp \cite{Tapp1220},
$g$ can be extended to a nonnegatively curved metric $\tilde{g}$ on $S^2 \times \mathbb{R}^4$ such that $(S^2 \times S^3, g)$ becomes
the distance sphere of radius $1$ about the soul. By Proposition \ref{weak}, we get a Riemannian submersion
$\pi: (S^2 \times S^3, g)\rightarrow (S^2,h)$, where $h$ is the induced metric on the soul $S^2$ from $\tilde{g}$.
This example shows that in Conjecture $1$ the assumption that $(M,g)$ has positive sectional curvature can $not$ be replaced by
$(M,g)$ has positive sectional curvature $almost$ everywhere.

\vskip0.1in

Riemannian submersions are also important in the study of compact Einstein manifolds, for example, see \cite{besse2007einstein}.
Our next theorem gives a complete classification
of Riemannian submersions with totally geodesic fibers from compact four-dimensional Einstein manifolds.
\begin{thm} \label{Einstein}
Suppose $\pi: (M^4,g)\rightarrow (N,h)$ is a Riemannian submersion, where $(M^4,g)$ is a compact four-dimensional Einstein manifold.
If all fibers of $\pi$ are totaly geodesic and have dimension $2$, then locally $\pi$  is the projection of a metric product $B^2(c)\times B^2(c)$ onto one of the factors,
where $B^2(c)$ is a two-dimensional compact manifold with constant curvature $c$.
\end{thm}

If the dimension of the fibers of $\pi$ is $1$ or $3$ (all fibers are not necessarily totally geodesic), then the Euler number of $M^4$ is zero. By a theorem of Berger \cite{Ber, hitchin1974compact}, $(M^4,g)$ must be flat.
 Hence by a theorem of Walschap \cite{walschap1992metric},  locally $\pi$  is the projection of a metric product $B^2(c)\times B^2(c)$ onto one of the factors.
\vskip0.1in

\noindent \textbf{Acknowledgment} This paper is a part of my Ph.D thesis at University of Notre Dame \cite{chen}.
The author would like to express gratitude to his advisor, Professor Karsten Grove, for many helpful discussions.
He also thanks Professor Anton Petrunin for discussing the proof of Theorem \ref{basic}. The author benefits a lot from his
"Exercises in orthodox geometry" \cite{petrunin2009exercises}.

\section{Preliminaries}

In this section we recall some definitions and facts on Riemannian submersions which will be used in this paper. We refer to \cite{O'Neil66} for more details.

\vskip0.1in

 Let $\pi: (M,g)\rightarrow (N,h)$ be a Riemannian submersion. Then $\pi$
induces an orthogonal splitting $TM=\mathcal{H}\oplus \mathcal{V}$, where $\mathcal{V}$ is tangent to the fibers and $\mathcal{H}$
is the orthogonal complement of $\mathcal{V}$. We write
$Z=Z^h + Z^v$ for the corresponding decomposition of $Z\in TM$. The O'Neill tensor $A$ is given by
$$A_X Y=(\nabla_X Y)^v=\frac{1}{2}([X,Y])^v,$$
where $X,Y\in \mathcal{H}$ and are $\pi$-related to some vector field on $N$, respectively.

\vskip0.1in

 Fix $X\in \mathcal{H}$, define $A^*_X$ by
$$A^*_X:\mathcal{V}\rightarrow \mathcal{H}$$
$$V\mapsto -(\nabla_X V)^h.$$
Then $A^*_X$ is the dual of $A_X$.

\vskip0.1in

 Define the mean curvature vector field $H$ of $\pi$ by
$$H=\sum_i (\nabla_{V_i} V_i)^h,$$
where $\{{V_i}\}_{i=1}^k$ is any orthonormal basis of $\mathcal{V}$ and $k=dim \mathcal{V}$.

\vskip0.1in

Define the mean curvature form $\omega$ of $\pi$ by
$$\omega(Z)=g(H,Z),$$
where $Z\in TM$.
It is clear that $i_V\omega =\omega (V)=0$ for any $V\in \mathcal{V}$.

\vskip0.1in

We say that a function $f$ defined on $M$ is basic if $f$ is constant along each fiber.
 A vector field $X$ on $M$ is basic if it is horizontal and is $\pi$-related to a vector field on $N$.
 In other words, $X$ is the horizontal lift of some vector field on $N$.
A differential form $\alpha$ on $M$ is called to be basic if and only $i_V \alpha=0$ and $\mathcal{L}_V \alpha=0$ for any $V\in\mathcal{V}$,
where $\mathcal{L}_V \alpha$ is the Lie derivative of $\alpha$.

\vskip0.1in

 The set of basic forms of $M$, denoted by $\Omega_b(M)$, constitutes a subcomplex
$$d:\Omega_b^r(M) \rightarrow \Omega_b^{r+1}(M)$$
of the De Rham complex $\Omega(M)$. The basic cohomology of $M$, denoted by $H^*_b(M)$, is defined to be the cohomology of
$(\Omega_b(M), d)$.

\begin{prop}\label{induce}
The inclusion map $i:\Omega_b(M)\rightarrow \Omega(M)$ induces an injective map
$$H^1_b(M)\rightarrow H^1_{DR}(M).$$
\end{prop}

\begin{proof}
See pages $33-34$, Proposition $4.1$ in \cite{tondeur1997geometry}.
\end{proof}

\section{Proof of  Proposition \ref{totally} and Theorem \ref{positive} }

We first give a proof of Proposition \ref{totally}.

\begin{proof}

Fix $p\in M$ and
choose $X_p$ to be any point in the unit sphere of $\mathcal{H}_p$.
Extend $X_p$ to be a unit basic vector field $X$. Since all fibers of $\pi$ are totally geodesic, by O'Neill's formula (\cite{O'Neil66}),
$K(X,V)=\|A^*_XV\|^2$ for any unit vertical vector field $V$.
Since $K(X,V)>0$ by assumption, we see that $A^*_XV \neq 0$ for any $V\neq 0$.
Let $v_1, v_2, \cdots v_k$ be any orthonormal basis of $\mathcal{V}_p$, where $k=dim (F_p)$ and $F_p$ is the fiber passing through $p$.
Then $A^*_X(v_1), A^*_X(v_2), \cdots A^*_X(v_k)$ are linearly independent and are 	
perpendicular to $X_p$. Since $X_p$ could be any point in the unit sphere of $\mathcal{H}_p$, then we get
$k$ linearly independent vector fields on the unit sphere of $\mathcal{H}_p$. By the definition of $\rho(dim N)$, we see that
$dim (F_p)=k \leq \rho(dim (N)) < \rho(dim (N))+1$.
\end{proof}

\begin{Remark} \label{topo} It would be very interesting to know whether one can replace
$dim (F)< dim (N)$ by $dim (F)< \rho(dim (N))+1$ in Conjecture $1$.
It would be the Riemannian analogue of Toponogov's Conjecture (page $1727$ in \cite{Rovenski})
and would imply that $dim (N)$ must be even (In fact, if $dim(N)$ is odd, then $\rho(dim (N))=0$. Hence $dim (F)< \rho(dim (N))+1$
implies $dim (F)=0$ and hence $\pi$ is trivial, contradiction).
In particular, there would be no Riemannian submersion with one-dimensional fibers
from even-dimensional manifolds with positive sectional curvature.
\end{Remark}

Let $(M^m, g)$ be an $m$-dimensional compact manifold with positive sectional curvature, $m\geq4$ and $(N^2,h)$ be a $2$-dimensional compact Riemannian manifold.
Now we are going to prove the following theorem which implies Theorem \ref{positive}.
\begin{thm} \label{basic}
There is no Riemannian submersion $\pi: (M^m,g) \rightarrow (N^2, h)$
such that 

\vskip0.1in

1. the Euler numbers of the fibers are nonzero and

\vskip0.1in

2. either $\|A\|$ or $H$ is basic.

\end{thm}

\begin{Remark} If Conjecture $1$ is true, then there would be no Riemannian submersion  $\pi: (M^m,g) \rightarrow (N^2, h)$, where 
 $(M^m,g)$ has  positive sectional curvature and  $m \geq 4$.
\end{Remark}

Before we prove Theorem \ref{basic}, we firstly show how to derive Theorem \ref{positive}. The proof is by contradiction.
Suppose there exists a nontrivial Riemannian submersion $\pi: (M^4,g)\rightarrow (N,h)$ such that either $\|A\|$ or $H$ is basic,
where $(M^4,g)$ is a compact four manifold with positive sectional curvature.
Since $(M^4,g)$ has positive sectional curvature, $H^1(M^4, \mathbb{R})=0$ by Bochner's vanishing theorem (\cite{Pet}, page $208$).
By Poincar$\acute{e}$ duality, $\chi(M^4)=2+b_2(M^4)$ is positive.
By a theorem of Hermann \cite{hermann1960sufficient}, $\pi$ is a locally trivial fibration.
Then $\chi (M^4)= \chi (N) \chi (F)$, where $F$ is any fiber of $\pi$.
It follows that $dim (N)=2$ and $\chi(F)$ is nonzero (hence all fibers have nonzero Euler numbers), which is a contradiction by
Theorem \ref{basic}.

\vskip0.1in

The  proof of  Theorem \ref{basic} is again by contradiction.  Suppose  $\pi: (M^m,g) \rightarrow (N^2, h)$ be a  Riemannian submersion
satisfying the conditions in Theorem \ref{basic}.
 By passing to its oriented double cover, we can assume that
  $N^2$ is oriented. The idea of the proof of  Theorem \ref{basic} is to construct a nowhere vanishing vector field (or line field)
  on some fiber of $\pi$,  which will  imply the Euler  numbers of the fibers are zero. Contradiction.

\vskip0.1in

 Since $(M,g)$ has positive sectional curvature, by a theorem of Walschap \cite{walschap1992metric},
  $\|A\|$ can not be identical to zero on $M$. Hence there exists $p\in M$ such that
 $\|A\|(p)\neq 0$.

\vskip0.1in

 If $\|A\|$ is basic, then $\|A\| \neq 0$ at any point on $F_p$, where $F_p$ is the fiber at $p$.
Let $X, Y$ be any orthonormal oriented basic vector fields in some open neighborhood of $F_p$. 
Then $\|A_XY\|^2=\frac{1}{2}\|A\|^2 \neq 0$ at any point on $F_p$.   Define a map $s$ by
$$s: F_p \rightarrow TF_p$$
$$ x \mapsto \frac{A_X Y}{\|A_X Y\|}(x).$$
Let $Z, W$ be another orthonormal oriented basic vector fields. Then $Z=a X+ b Y$ and $W= c X + d Y$, $ad-bc > 0$.
Then
$$A_Z W= (ad-bc) A_X Y.$$
Hence $s$ does not depend on the choice of $X, Y$. Then $s$ is a nowhere vanishing vector field on
$F_p$. Thus the Euler number of $F_p$ is zero. Contradiction.

\vskip0.1in

 If $H$ is basic, the construction of such nowhere vanishing vector field (or line field) is much more complicated.
 Under the assumption that $H$ is basic, we firstly construct a metric $\hat{g}$ on $M^m$ such that $\pi: (M^m,\hat{g})\rightarrow (N^2,h)$ is
 still a Riemannian submersion and all fibers are minimal submanifolds with respect to $\hat{g}$.
Of course, in general $\hat{g}$ can $not$ have positive sectional curvature everywhere.
However, the crucial point is that there exists some fiber $F_0$ such that $\hat{g}$ has positive sectional curvature at all points on $F_0$.
Pick any fiber $F_1$ which is close enough to $F_0$. Then using the classical variational argument, we construct a
continuous codimension one distribution on $F_1$. Thus the Euler number of $F_1$ is zero. Contradiction.

\vskip0.1in

 Now we are going to explain the proof of Theorem \ref{basic} in details. We firstly need the following lemmas:
\begin{lem} \label{closed}
Suppose $\omega$ is the mean curvature form of a Riemannian submersion from compact Riemannian manifolds.
If $\omega$ is a basic form, then it is a closed form.
\end{lem}

\begin{proof}
See page $82$ in \cite{tondeur1997geometry} for a proof.
\end{proof}

\begin{lem} \label{baobao}
Suppose $\pi: (M^m,g)\rightarrow (N,h)$ is a Riemannian submersion such that $H$ is basic, where $(M^m,g)$ is a compact
Riemannian manifold with positive sectional curvature.
Then there exists a metric $\hat{g}$ on $M^m$ such that
$\pi: (M^m,\hat{g})\rightarrow (N,h)$ is still a Riemannian submersion and all fibers are minimal submanifolds
with respect to $\hat{g}$. Furthermore, there exists some fiber $F_0$ such that $\hat{g}$ has
positive sectional curvature at all points on $F_0$.
\end{lem}

\begin{proof}
The idea is to use partial conformal change of metrics along the fibers, see also page $82$ in \cite{tondeur1997geometry}.
Let $\omega$ be the mean curvature form of $\pi$.
Since $H$ is basic, $\omega$ is a basic form.
Then $\omega$ is closed by Lemma \ref{closed}.
So $[\omega]$ defines a cohomological class in $H^1_b(M^m)$.
Because $(M^m,g)$ has positive sectional curvature,
$H^1_{DR}(M^m)=0$ by Bochner's vanishing theorem (\cite{Pet}, page $208$).
By Proposition \ref{induce}, we see that $H^1_b(M^m)=0$.
Then there exists a basic function
$f$ globally defined on $M^m$ such that $\omega =df$.
Define $\hat{f}=f-max_{p\in M^m} f(p)$. Then $max_{p\in M^m} \hat{f}(p)=0$
and $\omega =d\hat{f}$.
Let $\lambda=e^{\hat{f}}$ and define
$$\hat{g}=(\lambda ^\frac{2}{k}g_v)\oplus g_h,$$
where $k=dim (M^m)-dim (N)$, $g_v$ $/ g_h$ are the vertical $/$ horizontal components of $g$, respectively.

\vskip0.1in

Since the horizontal components of $g$ remains unchanged, $\pi: (M^m,\hat{g})\rightarrow (N,h)$ is still
a Riemannian submersion.
Now the mean curvature form $\hat{\omega}$ associated to
$\hat{g}$ is computed to be
$$\hat{\omega}=\omega - d log \lambda=0.$$
Hence all fibers of $\pi$ are minimal submanifolds with respect to $\hat{g}$.

\vskip0.1in

Let $\phi(p)=\lambda ^\frac{2}{k}(p), p\in M^m.$
Then
$$\hat{g}=(\phi g_v)\oplus g_h.$$
Note for any $p\in M^m$, $0< \phi(p)\leq 1$. Moreover, we have
$max_{p\in M^m}\phi(p)=1$. Let $p_0\in M^m$ such that $\phi(p_0)=1$ and $F_0$ be the fiber of $\pi$ passing through $p_0$.
Since $f$ is a basic function on $M^m$, $\phi$ is also basic. Then $\phi \equiv 1$ on $F_0$, which will play a crucial role in our argument below.
Of course, in general $\hat{g}$ can $not$ have positive sectional curvature everywhere. However,
by Lemma \ref{conformal} below, we see that $\hat{g}$ still has positive sectional curvature at all points on $F_0$.
(The reader should compare it to the following fact: Let $\hat{h}=e^{2f}h$ be a conformal change of $h$, where $h$ is a Riemannian metric on $M$ with positive sectional curvature. 
Then $\hat{h}$ still has positive sectional curvature at those points where $f$ attains its maximum value.)

\vskip0.1in

Indeed, by Lemma \ref{conformal} below, for any basic vector fields $X,Y$ and vertical vector fields $V,W$, we have
$$\hat{K}(X+V, Y+W)\|(X+V)\wedge (Y+W)\|^2=\hat{R}(X+V,Y+W,Y+W,X+V)$$
$$=R(X+V,Y+W,Y+W,X+V)+ (\phi-1) P (\nabla \phi, \phi, X, Y, V, W)$$
$$+ Q(\nabla \phi, \phi, X, Y, V, W)+[-g(W,W)g(\nabla_V \nabla \phi, X)$$
$$+ g(V,W) g( \nabla_W \nabla \phi, X)+ g(V,W) g( \nabla_V \nabla \phi, Y)$$
$$-g(V,V) g( \nabla_W \nabla \phi, Y)]+\frac{1}{2}[-Hess (\phi) (X,X) g(W,W)$$
$$+ 2 Hess (\phi) (X,Y) g(V,W)- Hess(\phi)(Y,Y) g(V,V)],$$
where $\hat{K}(X+V, Y+W)$ is the sectional curvature of the plane spanned by $X+V, Y+W$
with respect to $\hat{g}$ and $$\|(X+V)\wedge (Y+W)\|^2=\hat{g}(X+V, X+V)\hat{g}(Y+W, Y+W)- [\hat{g}(X+V, Y+W)]^2.$$
Moreover, $\nabla $ is the Levi-Civita connection and $Hess (\phi)$
is the Hessian of $\phi$ with respect to $g$.
Also $\hat{R} /R$ are the Riemannian curvature tensors with respect to $\hat{g}/ g$, respectively.
Furthermore, $P(\nabla \phi, \phi, X, Y, V, W)$, $ Q (\nabla \phi, \phi, X, Y, V, W)$ are two functions depending on $\nabla \phi, \phi, X, Y, V, W$
and $Q(0, \phi, X, Y, V, W)\equiv0$ (which will be very important for our purpose).

\vskip0.1in

Since $\phi \equiv 1= max_{p\in M^m} \phi (p)$ on $F_0$, we see that $\nabla \phi \equiv 0$ on $F_0$. Hence
$Q(\nabla \phi, \phi, X, Y, V, W) \equiv Q(0, \phi, X, Y, V, W) \equiv 0$ and $\nabla_V \nabla \phi  \equiv0,
\nabla_W \nabla \phi \equiv 0$ on $F_0$.
Then at any point on $F_0$, we have
$$\hat{R}(X+V,Y+W,Y+W,X+V)=R(X+V,Y+W,Y+W,X+V)$$
$$+\frac{1}{2}[-Hess (\phi) (X,X) g(W,W)+ 2 Hess (\phi) (X,Y) g(V,W)$$
$$- Hess(\phi)(Y,Y) g(V,V)].$$
On the other hand, let
\begin{gather*}
 A=
\begin{pmatrix}
Hess (\phi) (X,X) & Hess (\phi) (X,Y) \\
Hess (\phi) (X,Y) & Hess (\phi)(Y,Y),
\end{pmatrix},
B=
\begin{pmatrix}
g(W,W) & -g(V,W)\\
-g(V,W) & g(V,V)
\end{pmatrix}.
\end{gather*}
Then
$$\hat{R}(X+V,Y+W,Y+W,X+V)=R(X+V,Y+W,Y+W,X+V)+ \frac{1}{2}tr(-AB).$$
Since $\phi$ attains its maximum at any point on $F_0$, we see that $-A$ is nonnegative definite on $F_0$. It is easy to check that $B$ is also nonnegative definite. 
Hence $tr(-AB) \geq 0$ (although $-AB$ is not nonnegative definite if $AB \neq BA$). Since $g$ has positive sectional curvature
everywhere on $M^m$ by assumption, then at any point on $F_0$, we see that
$$\hat{R}(X+V,Y+W,Y+W,X+V) \geq R(X+V,Y+W,Y+W,X+V) >0.$$
Hence $\hat{g}$ still has positive sectional curvature at all points on $F_0$.
\end{proof}

\begin{lem}
\label{conformal}
Let $\pi: (M^m,g)\rightarrow (N,h)$ be a Riemannian submersion and $g=g_v\oplus g_h$, where
$g_v$ $/ g_h$ are the vertical $/$ horizontal components of $g$, respectively.
Suppose $\phi$ is a positive basic function defined on $M^m$.
Let $\hat{g}=(\phi$ $g_v)\oplus g_h$. Suppose $\hat{\nabla}/ \nabla $ are the Levi-Civita connections and
$\hat{R} /R$ are the Riemannian curvature tensors with respect to $\hat{g}/ g$, respectively. Moreover, let $Hess (\phi)$
be the Hessian of $\phi$ with respect to $g$. Then for any horizontal vector fields $X,Y$ ($X,Y$ are not necessarily basic vector fields) and vertical vector fields $V,W$, we have
$$ \hat{\nabla}_X Y=\nabla_X Y.$$
$$ \hat{\nabla}_V W=\nabla_V W-\frac{g(V,W)}{2}\nabla \phi+ (\phi-1)(\nabla_V W)^h.$$
$$ \hat{\nabla}_V X=\nabla_V X + \frac{g(X, \nabla \phi)}{2\phi}V + \frac{1-\phi}{2}\sum_{i=1}^n g([X,\varepsilon_i], V)\varepsilon_i.$$
$$ \hat{\nabla}_X V=\nabla_X V + \frac{g(X, \nabla \phi)}{2\phi}V + \frac{1-\phi}{2}\sum_{i=1}^n g([X,\varepsilon_i], V)\varepsilon_i,$$
where $\{{\varepsilon_i}\}_{i=1}^n$ is any orthonormal basis of the horizontal distribution with respect to $g$ and $n=dim (N)$.

\vskip0.1in

 Moreover, if $X,Y$ are $basic$ vector fields and $V,W$ are vertical vector fields, then
$$\hat{R}(X+V,Y+W,Y+W,X+V)=R(X+V,Y+W,Y+W,X+V)$$
$$+(\phi-1) P (\nabla \phi, \phi, X, Y, V, W) + Q(\nabla \phi, \phi, X, Y, V, W)$$
$$+[-g(W,W)g(\nabla_V \nabla \phi, X) + g(V,W) g( \nabla_W \nabla \phi, X)$$
$$+ g(V,W) g( \nabla_V \nabla \phi, Y) -g(V,V) g( \nabla_W \nabla \phi, Y)]$$
$$+\frac{1}{2}[-Hess (\phi) (X,X) g(W,W)+ 2 Hess (\phi) (X,Y) g(V,W)$$
$$- Hess(\phi)(Y,Y) g(V,V)],$$
where $P(\nabla \phi, \phi, X, Y, V, W), Q (\nabla \phi, \phi, X, Y, V, W)$ are two functions which depend on $\nabla \phi, \phi, X, Y, V, W$
and $Q(0, \phi, X, Y, V, W)\equiv0$.
\end{lem}

\begin{proof}
The proof is based on a lengthy computation and the following $Koszul's formula$:
$$2 \hat{g}(\hat{\nabla}_X Y, Z)=X (\hat{g}(Y,Z))+ Y (\hat{g}(Z,X))- Z (\hat{g}(X,Y))$$
$$+ \hat{g}([X,Y], Z)-\hat{g}([Y,Z],X)-\hat{g}([X,Z],Y).$$
We just prove the fourth-fifth equalities in Lemma \ref{conformal}, others are left to the readers.
In the computation below, we will use the following trick very often: If we encounter with anything like $\phi X$, we will rewrite
$\phi X= X + (\phi-1) X$. By rewriting it in this way, we can compare new curvature terms with odd terms. We  will also use the fact that $\phi$ is a basic function very often.

\vskip0.1in

Now let $X,Y$ be horizontal vector fields (not necessarily basic) and $\{{\varepsilon_i}\}_{i=1}^n$ be any orthonormal basis of the horizontal distribution with respect to $g$.
By $Koszul's formula$, we see that
$$2 \hat{g}(\hat{\nabla}_X V, \varepsilon_i)=X (\hat{g}(V,\varepsilon_i))+ V(\hat{g}(\varepsilon_i,X))-  \varepsilon_i(\hat{g}(X,V))$$
$$+ \hat{g}([X,V], \varepsilon_i)-\hat{g}([V,\varepsilon_i],X)-\hat{g}([X,\varepsilon_i],V)$$
$$=V (\hat{g}(\varepsilon_i,X))+ \hat{g}([X,V], \varepsilon_i)-\hat{g}([V,\varepsilon_i],X)-\hat{g}([X,\varepsilon_i],V).$$
Since $\hat{g}_h=g_h$ and $\hat{g}_v=\phi g_v$, we get
$$2 g(\hat{\nabla}_X V, \varepsilon_i)=V g(\varepsilon_i,X))+ g([X,V], \varepsilon_i)-g([V,\varepsilon_i],X)-\phi g([X,\varepsilon_i],V)$$
$$=V g(\varepsilon_i,X))+ g([X,V], \varepsilon_i)-g([V,\varepsilon_i],X)- g([X,\varepsilon_i],V)+ (1-\phi) g([X,\varepsilon_i],V).$$
By $Koszul's formula$ again, we see that
$$2 g(\nabla_X V, \varepsilon_i)=V g(\varepsilon_i,X))+ g([X,V], \varepsilon_i)-g([V,\varepsilon_i],X)- g([X,\varepsilon_i],V).$$
Then
$$2 g(\hat{\nabla}_X V, \varepsilon_i)=2 g(\nabla_X V, \varepsilon_i)+(1-\phi) g([X,\varepsilon_i],V).$$
Hence
$$(\hat{\nabla}_X V)^h=(\nabla_X V)^h + \frac{1-\phi}{2}\sum_{i=1}^n g([X,\varepsilon_i], V)\varepsilon_i.$$
Note that $\frac{1-\phi}{2}\sum_{i=1}^n g([X,\varepsilon_i], V)\varepsilon_i$ does not depend on the choice of $\{{\varepsilon_i}\}_{i=1}^n$.
By the similar argument above, we see that
$$(\hat{\nabla}_X V)^v=(\nabla_X V)^v + \frac{g(X, \nabla \phi)}{2\phi}V.$$
Hence
$$\hat{\nabla}_X V=\nabla_X V + \frac{g(X, \nabla \phi)}{2\phi}V + \frac{1-\phi}{2}\sum_{i=1}^n g([X,\varepsilon_i], V)\varepsilon_i.$$
The similar argument will also establish the first-third equalities in Lemma \ref{conformal}. We just mention that
in the proof of these equalities, the fact that $\phi$ is a basic function and hence $V \phi =0$ will be used very often.

\vskip0.1in

Now we are going to prove the fifth equality in  Lemma \ref{conformal}. In the following we always assume that $X,Y$ are basic vector fields. First of all, we have
$$\hat{R}(X+V,Y+W,Y+W,X+V)=\hat{R}(X,Y,Y,X)+\hat{R}(V,W,W,V)$$
$$+\hat{R}(X,W,W,X)+\hat{R}(Y,V,V,Y)+2\hat{R}(X,Y,Y,V)+2\hat{R}(Y,X,X,W)$$
$$+2\hat{R}(X,Y,W,V)+2\hat{R}(X,W,Y,V)+2\hat{R}(V,W,W,X)+2\hat{R}(W,V,V,Y).$$
Since $\hat{g}_h=g_h$, $(M^m, \hat{g})\rightarrow (N,h)$ is still a Riemannian submersion. Then by O'Neill's formula \cite{O'Neil66}, we have
$$I_0=\hat{R}(X,Y,Y,X)=R_N(X,Y,Y,X)-\frac{3}{4}\hat{g}([X,Y]^v,[X,Y]^v)$$
$$=R_N(X,Y,Y,X)-\frac{3}{4}g([X,Y]^v,[X,Y]^v)+\frac{3}{4}(1-\phi)g([X,Y]^v,[X,Y]^v)$$
$$=R(X,Y,Y,X)+\frac{3}{4}(1-\phi)g([X,Y]^v,[X,Y]^v),$$
where $R_N$ is the Riemannian curvature tensor of $(N,h)$.
On the other hand, by the first-fourth equalities in  Lemma \ref{conformal},
$$I_1=\hat{R}(V,W,W,V)=\hat{g}(\hat{\nabla}_V\hat{\nabla}_W W-\hat{\nabla}_W\hat{\nabla}_V W- \hat{\nabla}_{[V,W]}W, V)$$
$$=\phi g(\hat{\nabla}_V [\nabla_W W - \frac{1}{2}g(W,W)\nabla \phi + (\phi-1) (\nabla_W W)^h], V)$$
$$-\phi g(\hat{\nabla}_W [\nabla_V W - \frac{1}{2}g(V,W)\nabla \phi + (\phi-1) (\nabla_V W)^h], V)$$
$$-\phi g(\nabla_{[V,W]}W - \frac{1}{2}g([V,W],W)\nabla \phi + (\phi-1) (\nabla_{[V, W]}W)^h, V)$$
$$=\phi g(\hat{\nabla}_V (\nabla_W W)- \hat{\nabla}_W (\nabla_V W)-\nabla_{[V,W]}W, V)$$
$$- \frac{1}{2}\phi [g(W,W) g(\hat{\nabla}_V \nabla \phi, V)-g(V,W) g(\hat{\nabla}_W \nabla \phi, V)]$$
$$+(\phi-1) \tilde{P_1}(\nabla \phi, \phi, X, Y, V, W) + \tilde{Q_1}(\nabla \phi, \phi, X, Y, V, W).$$
$$=\phi g(\hat{\nabla}_V (\nabla_W W)^v + \hat{\nabla}_V(\nabla_W W)^h, V)-\phi g(\hat{\nabla}_W (\nabla_V W)^v + \hat{\nabla}_W(\nabla_V W)^h, V)$$
$$- \phi g(\nabla_{[V,W]}W, V)- \frac{1}{2}\phi [g(W,W) g(\hat{\nabla}_V \nabla \phi, V)-g(V,W) g(\hat{\nabla}_W \nabla \phi, V)]$$
$$+(\phi-1) \tilde{P_1}(\nabla \phi, \phi, X, Y, V, W) + \tilde{Q_1}(\nabla \phi, \phi, X, Y, V, W).$$
Since $\phi$ is a basic function, $g(\nabla \phi, V)=V \phi=0$. Hence $\nabla \phi$ is a horizontal vector field. Then by the first-fourth equalities in  Lemma \ref{conformal},
 we see that
 $$g(\hat{\nabla}_W \nabla \phi, V)=-g(\nabla_W V, \nabla \phi)+\frac{g(\nabla \phi, \nabla \phi)}{2 \phi}g(W,V),$$
and
$$I_1=\hat{R}(V,W,W,V)=\phi R(V,W,W,V)$$
$$+ (\phi-1) \check{P_1} (\nabla \phi, \phi, X, Y, V, W) + \check{Q_1}(\nabla \phi, \phi, X, Y, V, W)$$
$$=R(V,W,W,V)+(\phi-1)R(V,W,W,V)$$
$$+ (\phi-1) \check{P_1}(\nabla \phi, \phi, X, Y, V, W) +\check{Q_1}(\nabla \phi, \phi, X, Y, V, W)$$
$$=R(V,W,W,V)+(\phi-1) P_1 (\nabla \phi, \phi, X, Y, V, W) + Q_1(\nabla \phi, \phi, X, Y, V, W),$$
where $P_1(\nabla \phi, \phi, X, Y, V, W), Q_1 (\nabla \phi, \phi, X, Y, V, W)$ are two functions depending on $\nabla \phi, \phi, X, Y, V, W$
and $Q_1(0, \phi, X, Y, V, W)\equiv0$.

\vskip0.1in

 Since $X$ is a basic vector field, $[X,W]$ is vertical. Hence by the first-fourth equalities in  Lemma \ref{conformal},
$$I_2=\hat{R}(X,W,W,X)=\hat {g}(\hat{\nabla}_X \hat{\nabla}_W W- \hat{\nabla}_W \hat{\nabla}_X W- \hat{\nabla}_{[X,W]}W, X)$$
$$=g(\hat{\nabla}_X[\nabla_W W - \frac{1}{2}g(W,W)\nabla \phi + (\phi-1) (\nabla_W W)^h], X)$$
$$-g(\hat{\nabla}_W[\nabla_X W+\frac{g(X, \nabla \phi)}{2\phi}W + \frac{1-\phi}{2}\sum_{i=1}^n g([X,\varepsilon_i], W)\varepsilon_i], X)$$
$$-g(\nabla_{[X,W]}W - \frac{1}{2}g([X,W],W)\nabla \phi + (\phi-1) (\nabla_{[X, W]}W)^h, X)$$
$$= R(X,W,W,X)+ (\phi-1) P_2 (\nabla \phi, \phi, X, Y, V, W)$$
$$+ Q_2(\nabla \phi, \phi, X, Y, V, W)-\frac{1}{2} Hess (\phi)(X,X)g(W,W),$$
where $P_2(\nabla \phi, \phi, X, Y, V, W), Q_2 (\nabla \phi, \phi, X, Y, V, W)$ are two functions depending on $\nabla \phi, \phi, X, Y, V, W$
and $Q_2(0, \phi, X, Y, V, W)\equiv0$.

\vskip0.1in

 By the similar argument, we see that
$$I_3=\hat{R}(Y,V,V,Y)=R(Y,V,V,Y)+(\phi-1) P_3 (\nabla \phi, \phi, X, Y, V, W)$$
$$+ Q_3(\nabla \phi, \phi, X, Y, V, W)-\frac{1}{2} Hess (\phi)(Y,Y)g(V,V).$$
$$I_4=\hat{R}(X,Y,Y,V)=R(X,Y,Y,V)+(\phi-1) P_4 (\nabla \phi, \phi, X, Y, V, W)$$
$$+ Q_4(\nabla \phi, \phi, X, Y, V, W).$$
$$I_5=\hat{R}(Y,X,X,W)=R(Y,X,X,W)+(\phi-1) P_5 (\nabla \phi, \phi, X, Y, V, W)$$
$$+ Q_5(\nabla \phi, \phi, X, Y, V, W).$$
$$I_6=\hat{R}(X,Y,W,V)=R(X,Y,W,V)+(\phi-1) P_6(\nabla \phi, \phi, X, Y, V, W)$$
$$+ Q_6(\nabla \phi, \phi, X, Y, V, W).$$
$$I_7=\hat{R}(X,W,Y,V)=R(X,W,Y,V)+(\phi-1) P_7 (\nabla \phi, \phi, X, Y, V, W)$$
$$+ Q_7(\nabla \phi, \phi, X, Y, V, W)+\frac{1}{2} Hess(\phi)(X,Y)g(V,W).$$
$$I_8=\hat{R}(V,W,W,X)=R(V,W,W,X)+(\phi-1) P_8 (\nabla \phi, \phi, X, Y, V, W)$$
$$+ Q_8(\nabla \phi, \phi, X, Y, V, W)+\frac{1}{2}g(V,W)g(\nabla_W \nabla \phi, X)-\frac{1}{2}g(W,W)g(\nabla_V \nabla \phi, X).$$
$$I_9=\hat{R}(W,V,V,Y)=R(W,V,V,Y)+(\phi-1) P_9 (\nabla \phi, \phi, X, Y, V, W)$$
$$+ Q_9(\nabla \phi, \phi, X, Y, V, W)+\frac{1}{2}g(V,W)g(\nabla_V \nabla \phi, Y)-\frac{1}{2}g(V,V)g(\nabla_W \nabla \phi, Y),$$
where $P_i(\nabla \phi, \phi, X, Y, V, W), Q_i(\nabla \phi, \phi, X, Y, V, W)$ are two functions depending on $\nabla \phi, \phi, X, Y, V, W$
and $Q_i(0, \phi, X, Y, V, W)\equiv0$, $i=3,4,\cdots 9.$
Hence
$$\hat{R}(X+V,Y+W,Y+W,X+V)=I_0+I_1+I_2+I_3+ 2\sum_{i=4}^9I_i$$
$$=R(X+V,Y+W,Y+W,X+V)+(\phi-1) P (\nabla \phi, \phi, X, Y, V, W)$$
$$ + Q(\nabla \phi, \phi, X, Y, V, W)+[-g(W,W)g(\nabla_V \nabla \phi, X) + g(V,W) g( \nabla_W \nabla \phi, X)$$
$$+ g(V,W) g( \nabla_V \nabla \phi, Y) -g(V,V) g( \nabla_W \nabla \phi, Y)]$$
$$+\frac{1}{2}[-Hess (\phi) (X,X) g(W,W)+ 2 Hess (\phi) (X,Y) g(V,W)$$
$$- Hess(\phi)(Y,Y) g(V,V)],$$
where $P(\nabla \phi, \phi, X, Y, V, W), Q (\nabla \phi, \phi, X, Y, V, W)$ are two functions which depend on $\nabla \phi, \phi, X, Y, V, W$
and $Q(0, \phi, X, Y, V, W)\equiv0$.
\end{proof}

\noindent \textbf{Proof of Theorem \ref{basic}:}
\begin{proof}
We prove it by contradiction. We already proved it if $\|A\|$ is basic. Hence it suffices to show it if $H$ is basic. We prove it by contradiction.
Let  $\pi: (M^m,g) \rightarrow (N^2, h)$ be a  Riemannian submersion such that $H$ is basic and the fibers have nonzero Euler numbers, where $(M^m,g)$ has  positive sectional curvature and  $m \geq 4$.
By Lemma \ref{baobao}, there exists a metric $\hat{g}$ on $M^m$ such that
$\pi: (M^m,\hat{g})\rightarrow (N^2,h)$ is still a Riemannian submersion and
all fibers of $\pi$ are minimal submanifolds with respect to $\hat{g}$.
Furthermore, there exists some fiber $F_0$ such that $\hat{g}$
has positive sectional curvature at all points in $F_0$.
Let $r$ be a fixed positive number such that the normal exponential map of $F_0$ is a
diffeomorphism when restricted to the tubular neighborhood of $F_0$ with radius $r$.
By continuity of sectional curvature, there exists $\epsilon$, $0<\epsilon <r$ such that
$\hat{g}$ has positive sectional curvature at the $\epsilon$ neighborhood of $F_0$.
Choose another fiber $F_1$ such that $0<\hat{d}(F_0,F_1)<\epsilon$, where $\hat{d}(F_0,F_1)$ is the distance between $F_0$ and $F_1$ with respect to
$\hat{g}$. Since $\pi: (M^m,\hat{g})\rightarrow (N^2,h)$ is a Riemannian submersion, $F_0$ and $F_1$ are equidistant. On the other
hand, since $0<\hat{d}(F_0,F_1)<\epsilon$, then
for any point $q\in F_1$, there is a $unique$ point $p\in F_0$ such that $\hat{d}(p,q)=\hat{d}(F_0, F_1)$. Let $L=\hat{d}(p,q)$
and $\gamma: [0,L]\rightarrow M^m, \gamma (0)=p, \gamma(L)=q$ be the $unique$ minimal geodesic with unit speed
realizing the distance between $p$ and $q$.
Let $V\subseteq T_q (M^m)$ be the subspace of vectors $v=X(L)$ where $X$ is a parallel field along
 $\gamma$ such that $X(0)\in T_p(F_0)$. Then
 $$dim (V \cap T_q(F_1))=dim (V) + dim (T_q(F_1)) - dim (V + T_q(F_1))$$
 $$\geq (m-2)+(m-2)-(m-1)=m-3.$$
 We claim that $dim (V \cap T_q(F_1))=m-3$. If not, then $dim (V \cap T_q(F_1))=m-2.$
 Let $X_i, i=1,\cdots m-2,$ be orthonormal parallel fields along
 $\gamma$ such that $X_i(0)\in T_p(F_0), X_i(L)\in T_q(F_1)$. For each $i$, choose a variation $f_i (s,t)$ of $\gamma$
  such that $f_i(s,0)\in F_0, f_i(s,L)\in F_1$ for small $s$ and
  $\frac{\partial f_i(0,t)}{\partial s} =X_i(t)$.
 By construction, $\dot{X_i}(t)=\hat{\nabla}_{\dot{\gamma}}X_i(t)=0$ for all $t$, where $\hat{\nabla}$ is
 the Levi-Civita connection with respect to $\hat{g}$.
 By the second variation formula, for $i=1,\cdots m-2,$ we have
 $$\frac{1}{2}\frac{d^2 E_i(s)}{ds^2}_{|s=0}=\int_0^L ( \hat{g}(\dot{X_i}, \dot{X_i})- \hat{R}(X_i,\dot{\gamma},\dot{\gamma},X_i))dt$$
 $$+ \hat{g}(\hat{B}_1(X_i,X_i), \dot{\gamma})(L)-\hat{g}(\hat{B}_0(X_i,X_i), \dot{\gamma})(0)$$
 $$=-\int_0^L \hat{R}(X_i,\dot{\gamma},\dot{\gamma},X_i)dt
 + \hat{g}(\hat{B}_1(X_i,X_i), \dot{\gamma})(L)-\hat{g}(\hat{B}_0(X_i,X_i), \dot{\gamma})(0),$$
 where $E_i(s)=\int_0^L \hat{g}(\frac{\partial f_i(s,t)}{\partial t}, \frac{\partial f_i(s,t)}{\partial t}) dt $,
 $\hat{R}$ is the curvature tensor of $\hat{g}$ and $\hat{B}_j$ is the second fundamental form of $F_j$
 with respect to $\hat{g}$, $j=0,1$.

\vskip0.1in

 Since $F_0$ and $F_1$ are minimal submanifolds in $(M^m,\hat{g})$, we have
 $$\sum_{i=1}^{m-2}\hat{B}_j(X_i,X_i)=0, j=0,1.$$
 Then
 $$\frac{1}{2}\sum_{i=1}^{m-2}\frac{d^2 E_i(s)}{ds^2}_{|s=0}=-\sum_{i=1}^{m-2}\int_0^L \hat{R}(X_i,\dot{\gamma},\dot{\gamma}, X_i)dt.$$
 Since $\hat{g}$ has positive sectional curvature at the $\epsilon$ neighborhood of $F_0$ and $0<\hat{d}(F_0,F_1)<\epsilon$,
 we see that
 $\hat{R}(X_i,\dot{\gamma},\dot{\gamma}, X_i)<0.$
 Hence
 $$\frac{1}{2}\sum_{i=1}^{m-2}\frac{d^2 E_i(s)}{ds^2}_{|s=0}<0.$$
Then there exists some $i_0$ such that $\frac{d^2 E_{i_0}(s)}{ds^2}_{|s=0}<0$, which contradicts that
 $\gamma$ is a minimal geodesic realizing the distance between $F_0$ and $F_1$.
 So $dim (V \cap T_q(F_1))=m-3$. Since $dim (T_q(F_1))=m-2$, then
 $V \cap T_q(F_1)$ is a codimension one subspace of $T_q(F_1)$.
 Since $q$ is arbitrary on $F_1$, by doing the same construction as above for any $q$,
 then we get a continuous codimension one distribution on $F_1$.
Thus the Euler number of $F_1$ is zero. Contradiction.
 \end{proof}

\section{Proof of theorem \ref{Einstein}}

In this section we prove Theorem \ref{Einstein}.
Suppose $\pi:(M^4,g)\rightarrow (N^2,h)$ is a Riemannian submersion with totally geodesic fibers, where
$(M^4,g)$ is a compact four-dimensional Einstein manifold. We are going to show that the $A$ tensor of $\pi$ vanishes and then
locally $\pi$ is the projection of a metric product onto one of the factors. We firstly need the following lemmas:

\begin{lem}\label{isofiber}
Let  $\pi$ be a Riemannian submersion with totally geodesic fibers from compact Riemannian manifolds, then all fibers are isometric to each other.
\end{lem}

\begin{proof}

See \cite{hermann1960sufficient}.

\end{proof}

\begin{lem} \label{besse}
Suppose $\pi:(M^4,g)\rightarrow (N^2,h)$ is a Riemannian submersion with totally geodesic fibers, where
$(M^4,g)$ is a compact four-dimensional Einstein manifold.
Let $c_1$, $c_2$ be the sectional curvature of $(F^2,g_{|F^2})$ and $(N^2,h)$, respectively, where
$g_{|F^2}$ is the restriction of $g$ to the fibers $F^2$.
Let $Ric (g)=\lambda g$ for some $\lambda$.
Then

\vskip0.1in

$(i)$  $2c_1+\|A\|^2=2\lambda;$

\vskip0.1in

$(ii)$  $2c_2 \circ \pi -2\|A\|^2=2\lambda;$

\vskip0.1in

$(iii)$ $\|A\|^2=\frac{2}{3}(c_2 \circ \pi -c_1),$

\vskip0.1in

where $\|A\|^2=\|A^*_XU\|^2+\|A^*_XV\|^2+\|A^*_YU\|^2+\|A^*_YV\|^2$.
Here $X,Y/ U,V$ is an orthonormal basis of $\mathcal{H}/\mathcal{V}$, respectively.

\end{lem}

\begin{proof}
See page $250$, Corollary $9.62$ in \cite{besse2007einstein}. For completeness, we give a proof here.

\vskip0.1in

Let $U,V$ / $X,Y$ are orthonormal basis of $\mathcal{V}$ / $\mathcal{H}$, respectively. Then  by O'Neill's formula (\cite{O'Neil66}) , we have
$$\lambda = Ric (U,U)=c_1 + \|A^*_XU\|^2 + \|A^*_YU\|^2;$$
$$\lambda=Ric(V,V)=c_1+\|A^*_XV\|^2+\|A^*_YV\|^2;$$
$$\lambda=Ric(X,X)=c_2 \circ \pi- 3\|A_XY\|^2 +\|A^*_XU\|^2+\|A^*_XV\|^2;$$
$$\lambda=Ric(Y,Y)=c_2 \circ\pi- 3\|A_XY\|^2 +\|A^*_YU\|^2+\|A^*_YV\|^2.$$
On the other hand, by direct calculation, we see that $2 \|A_XY\|^2=\|A\|^2$. Hence
$$2c_1+\|A\|^2=2\lambda;$$
$$ 2c_2 \circ \pi -2\|A\|^2=2\lambda;$$
$$\|A\|^2=\frac{2}{3}(c_2 \circ \pi-c_1).$$

\end{proof}

By Lemmas \ref{isofiber} and \ref{besse}, we see that $c_1,\|A\|$ are constants on $M^4$ and $c_2$ is a constant on $N^2$.

\vskip0.1in

Fix $p\in M^4$. Locally we can always choose basic vector fields $X,Y$ such that $X,Y$ is an orthonormal basis of the horizontal distribution.
At point $p$, since the image of $A^*_X$ is perpendicular to $X$ and $dim \mathcal{V}=dim \mathcal{H}=2$,
$A^*_X$ must have nontrivial kernel. Then there exists some $v\in \mathcal{V} $ such that $\|v\|=1$ and $A^*_X(v)=0$. Extend $v$ to be a
local unit vertical vector field $V$ and choose $U$ such that $U,V$ is a local orthonormal basis of $\mathcal{V}$.
\begin{lem} \label{tensor}
$$A^*_XV(p)=0;$$
$$A^*_YV(p)=0.$$
\end{lem}
\begin{proof}
We already see $A^*_XV(p)=A^*_{X,p}(v)=0$. On the other hand, at point $p$, we have
$$A^*_YV=g(A^*_YV, X)X=-g(\nabla_YV, X)X$$
$$=g(V, \nabla_YX)X=g(V, A_YX)X$$
$$=-g(V, A_XY)X=-g(V, \nabla_XY)X$$
$$=g(\nabla_XV, Y)X=-g(A^*_XV, Y)X=0.$$
\end{proof}
Since all fibers of $\pi$ are totally geodesic, by O'Neill's formula (\cite{O'Neil66}), we see that $K(X,U)=\|A^*_XU\|^2$ .
Because $(M^4,g)$ is Einstein, at point $p$, we have
$$\lambda=Ric(U,U)=c_1+\|A^*_XU\|^2+\|A^*_YU\|^2;$$
$$\lambda=Ric(V,V)=c_1+\|A^*_XV\|^2+\|A^*_YV\|^2;$$
Combined with Lemma \ref{tensor}, we see that $\lambda=c_1$ and $\|A^*_XU\|^2(p)=0$,
$\|A^*_YU\|^2(p)=0$.
Then $\|A\|^2(p)=0$. Hence $\|A\|^2\equiv 0$ on $M^4$ and $c_1=c_2$. Let $c=c_1=c_2$.
Then locally $\pi$ is the projection of a metric product $B^2(c)\times B^2(c)$ onto one of the factors,
where $B^2(c)$ is a two-dimensional compact manifold with constant curvature $c$.

\section{Conjecture 1 and the Weak Hopf Conjecture}

In this section we point out several interesting corollaries of Conjecture $1$.

\vskip0.1in

Suppose $(E,g)$ is a complete, open Riemannian manifold with nonnegative sectional curvature. By
a well known theorem of Cheeger and Gromoll \cite{cheeger1972structure},
$E$ contains a compact totally geodesic submanifold $\Sigma$, called the soul,
such that $E$ is diffeomorphic to the normal bundle of $\Sigma$.
Let $\Sigma_r$ be the distance sphere to $\Sigma$ of radius $r$.
Then for small $r>0$, the induced metric on $\Sigma_r$ has nonnegative sectional curvature
by a theorem of Guijarro and Walschap \cite{guijarro2000metric}. In \cite{Gro2003}, Gromoll and Tapp proposed the following conjecture:

\vskip0.1in

\noindent \textbf{Weak Hopf Conjecture} {\itshape Let $k\geq 3$. Then for $any$ complete metric with nonnegative sectional curvature
 on $S^n \times \mathbb{R}^k$, the induced metric on the boundary of a small metric tube about the soul can $not$ have positive sectional curvature.}

\vskip0.1in

 The case $n=2, k=3$ is of particular interest since the metric tube of the soul is diffeomorphic to $S^2 \times S^2$.

\vskip0.1in

Recall that a map between metric spaces $\sigma: X \rightarrow Y $ is a submetry if for all $
x\in X$ and $ r\in [0, r(x)]$ we have that $f(B(x, r)) = B(f(x), r)$, where $B(p, r)$ denotes the open
metric ball centered at $p$ of radius $x$ and $r(x)$ is some positive continuous function. If both $X$ and $Y$ are
Riemannian manifolds, then $\sigma$ is a Riemannian submersion of class $C^{1,1}$
by a theorem of Berestovskii and Guijarro \cite{berestovskii2000metric}.

\begin{prop}\label{weak}
 Suppose $\Sigma$ is a soul of $(E,g)$, where $(E,g)$ is a complete, open Riemannian manifold with nonnegative sectional curvature. If
  the induced metric on $\Sigma_r$ has positive sectional curvature at some point for some $r>0$,
 then there is a Riemannian submersion from $\Sigma_r$ to $\Sigma$ with fibers $S^{l-1}$, where
 $l=dim(E) - dim(\Sigma).$
\end{prop}

\begin{proof}
In fact, by a theorem of Guijarro and Walschap in \cite{guijarro2008dual}, if $\Sigma_r$ has positive sectional curvature at some point,
 the normal holonomy group of $\Sigma$ acts transitively on $\Sigma_r$. By Corollary 5 in \cite{wilking2007duality}, we get a submetry
 $\pi: (E,g)\rightarrow \Sigma \times [0, +\infty)$ with fibers $S^{l-1}$, where $\Sigma \times [0, +\infty)$ is endowed with the product metric.
Then $\pi: (\pi^{-1}(\Sigma \times (0, +\infty)),g)\rightarrow \Sigma \times (0, +\infty)$ is also a submetry. By a theorem of Berestovskii and Guijarro in \cite{berestovskii2000metric},
$\pi$ is a $C^{1,1}$ Riemannian submersion. Then $\Sigma_r=\pi^{-1}(\Sigma \times \{{r}\})$ and
$\pi: \Sigma_r \rightarrow \Sigma$ is also a $C^{1,1}$ Riemannian submersion with fibers $S^{l-1}$, where $\Sigma_r$ is endowed with
the induced metric from $(E,g)$.
\end{proof}

\begin{prop}
When $k>n$, Conjecture $1$ implies Weak Hopf Conjecture.
\end{prop}

\begin{proof}
Suppose for some complete metric $g$ on $S^n \times \mathbb{R}^k$
with nonnegative sectional curvature, the induced metric on $\Sigma_r$ has positive sectional curvature
for some $r>0$, where $\Sigma$ is a soul.
Since $S^n \times \mathbb{R}^k$ is diffeomorphic to the normal bundle of $\Sigma$, we see that $\Sigma$ is a homotopy sphere and
$dim (\Sigma)=n$.
By Proposition \ref{weak}, we get a Riemannian submersion from $\Sigma_r$ to $\Sigma$ with fibers $S^{k-1}$, where
$\Sigma_r$ is endowed with the induced metric from $g$ and hence has positive sectional curvature.
Since $k>n$, we see $k-1 \geq n$, which is impossible if Conjecture $1$ is true for $C^{1,1}$ Riemannian submersions.
\end{proof}

\begin{Remark}
If Remark \ref{topo} in section $3$ is true, then by Proposition \ref{weak} again, any small metric tube about the soul can $not$ have positive sectional curvature
when the soul is odd-dimensional. This would give a solution to a question asked by K. Tapp in \cite{Tapp2012}.
\end{Remark}

{}
Department of Mathematics\\
University of Notre Dame\\
Notre Dame, Indiana, 46556.\\
E-mail address: {\em xychen100@gmail.com}

\end{document}